\documentclass[12pt]{amsart}
                                    
\usepackage{amsmath, graphicx}                                                                                                                                                                    
\usepackage{color}
\usepackage{amsthm}
\usepackage{amssymb}
\usepackage{amscd}
\usepackage{hyperref}
\usepackage{tikz}

\usepackage{enumerate}
\usetikzlibrary{matrix,arrows}

\newtheorem{theorem}{Theorem}[section]
\newtheorem{proposition}[theorem]{Proposition}
\newtheorem{corollary}[theorem]{Corollary}

\theoremstyle{definition}
\newtheorem{definition}[theorem]{Definition}

\newtheorem{example}[theorem]{Example}
\newtheorem{remark}[theorem]{Remark}      
\newtheorem{observation}[theorem]{Observation}                         
\newtheorem*{acknowledgement}{Acknowledgments}

\DeclareMathOperator{\Tor}{Tor}
\DeclareMathOperator{\type}{mtype}

\DeclareMathOperator{\codim}{codim}

\DeclareMathOperator{\td}{td}
\DeclareMathOperator{\pd}{pd}

\DeclareMathOperator{\lk}{link}

\title{The  Type Defect of a Simplicial Complex}

\author[H.~Dao]{Hailong Dao}     
\address{Hailong Dao\\     
	Department of Mathematics\\     
	University of Kansas\\     
	405 Snow Hall \\
1460 Jayhawk Blvd. \\
Lawrence, KS 66045-7594  
	USA}     
\email{hdao@math.ku.edu}
\author[J.~Schweig]{Jay Schweig}
\address{Jay Schweig\\
	Department of Mathematics\\                   
	Oklahoma State University\\     
	401 MSCS \\
	Stillwater\\
	OK \ 74078-1058\\     
	USA}    
\email{jay.schweig@okstate.edu}


\begin{document}
\maketitle

%\begin{abstract}
%Let $k$ be a field and  $\Delta$ be a simplicial complex with $n$ vertices. In this work we define and study the type defect of $\Delta$, denoted by $\td(\Delta)$. It is a modification of the $c$-Betti number of the resolution of $S/I_{\Delta}$ over $S=k[x_1,\dots,x_n]$ where $c$ is the codimension of $I_{\Delta}$. We show that this invariant enjoys surprisingly nice properties. For example, $\Delta$ is Cohen-Macaulay if  $\td(\Delta)\leq 0$. On the other hand, if $\Delta$ is a graph, then it is chordal if and only if $\td(\Delta')\geq 0$ for each induced subgraph $\Delta'$. The type defect also behaves very well when one glues complexes along a face, which we exploit to study classes of $\Delta$ such that $\td(\Delta')=0$ for all connected induced subcomplex $\Delta'$. These can be viewed as  generalizations of trees. In the second half of the paper, we extend part of the result about chordality to higher dimensions by proving sharp lower bounds for most Betti numbers of an ideal with linear resolution, monomial or not. We classify when equalities occur.  
%\end{abstract}

\begin{abstract}
Fix a field $k$. When $\Delta$ is a simplicial complex on $n$ vertices with Stanley-Reisner ideal $I_\Delta$, we define and study an invariant called the \emph{type defect} of $\Delta$.  Except when $\Delta$ is a single simplex, the type defect of $\Delta$, $\td(\Delta)$, is the difference $ \dim \Tor_c^S(S/ I_\Delta,k) - c$, where $c$ is the codimension of $\Delta$ and $S = k[x_1, \ldots x_n]$.  We show that this invariant admits surprisingly nice properties.  For example, it is well-behaved when one glues two complexes together along a face. Furthermore,  $\Delta$ is Cohen-Macaulay if  $\td(\Delta) \leq 0$.  On the other hand, if $\Delta$ is a simple graph (viewed as a one-dimensional complex), then $\td(\Delta') \geq 0$ for every induced subgraph $\Delta'$ of $\Delta$ if and only if $\Delta$ is chordal.  Requiring connected induced subgraphs to have type defect zero allows us to define a class of graphs that we call \emph{treeish}, and which we generalize to simplicial complexes.  We then extend some of our chordality results to higher dimensions, proving sharp lower bounds for most Betti numbers of ideals with linear resolution, and classifying when equalities occur.  As an application, we prove sharp lower bounds for Betti numbers of graded ideals (not necessarily monomial) with linear resolution.

\end{abstract}
\thispagestyle{empty}

\section{Introduction}

Let $k$ be a  field, let $S = k[x_1, x_2, \ldots, x_n]$, and let $I$ be a homogenous ideal such that $S/I$ is Cohen-Macaulay.  The \emph{type} of the quotient $S/I$ is the total Betti number $b_c(S/I) :=  \dim \Tor_c^S(S/ I,k)$, where $c$ is the codimension of $I$.  That is, the type of $S/I$ is the dimension of the free module at position $c$ in a minimal free resolution of $S/I$ over $S$.  In the local setting, this invariant is rather classical. In this paper we study some variants of this concept for (not necessarily) Cohen-Macaulay square-free monomial ideals. Roughly speaking, we shall show that these invariants behave very well with respect to the ``gluing" of simplicial complexes, and that they are strongly linked to important and much-studied properties of complexes such as chordality, linear resolution, or Cohen-Macaulayness.     

Let $\Delta$ be a $(d-1)$ dimensional simplicial complex with $n$ vertices. Throughout, let $c=n-d$, the codimension of $\Delta$.  Our key definitions are as follows. 

\begin{definition} We define the {\it modified type} of $\Delta$ to be: 
\[
\type(\Delta) = \dim \Tor_c^S(S/ I_\Delta,k) - \dim \Tor_c^S(S/ I_\Delta,k)_c
\]
where $I_\Delta$ is the Stanley-Reisner ideal of $\Delta$ and $M_c$ denotes the degree $c$ component of a graded module $M$.  
We also define the \emph{type defect} of $\Delta$, $\td(\Delta)$, to be the difference between its modified type and codimension:
\[
\td(\Delta) = \type(\Delta) - c. 
\]
\end{definition}

Except in the case when $\Delta$ has one facet, $\type(\Delta)$ equals $b_c(S/I_\Delta)$, the total Betti number of $S/I$ with index $c$, see Corollary \ref{basic}.  However, this small modification (see Section \ref{s:background}) actually makes the results and proofs much more elegant. We summarize some of the key statements below (see Theorems \ref{typefact}, \ref{chordal} and Corollary \ref{secondglue} for details):

\begin{theorem}\label{intro1}
Let $\Delta$ be a $(d-1)$-simplicial complex with $n$ vertices. 

\begin{enumerate}

\item (Cohen-Macaulayness) If $\td(\Delta) \leq 0$ then $\Delta$ is Cohen-Macaulay (over $k$). 

\item (Chordality) Assume that $\Delta$ is a graph with vertex set $V$. Then $\Delta$ is chordal if and only if  $\td(\Delta|_W) \geq 0$ for any $W \subseteq V$.  

\item (Gluing) Suppose that $\Delta$ is obtained by gluing two Cohen-Macaulay complexes $\Delta_1$ and $\Delta_2$ of dimension $(d-1)$ along an $(\ell-1)$-dimensional face. Then 
\[
\td(\Delta) = \td(\Delta_1) + \td(\Delta_2) -d+l + \binom{n-\ell}{n-d+1}.
\]
In  particular if $\ell \geq d-1$, then $$\td(\Delta) = \td(\Delta_1) + \td(\Delta_2).$$
\end{enumerate}
\end{theorem}

In the second half of this work, we focus on extending the chordality part of Theorem \ref{intro1} to complexes of higher dimensions. The main difficulty here is that there are many proposed definitions of chordality in general (for an incomplete list, see \cite{ans}, \cite{bhpz}, \cite{emt}, \cite{hvt},  \cite{russ}). Thus we focus our attention on complexes with {\it linear resolutions}, a property that often appears whatever definition of chordality one uses. Here we are able to prove that for such a complex, $\td(\Delta)\geq 0$. In fact, we prove a number of lower bounds for {\it all} the Betti numbers of $S/I_{\Delta}$ up to the codimension and characterize when any of these inequalities become equality. The full result reads:

\begin{theorem}\label{intro2}
Let $\Delta$ be a simplicial complex  such that $I=I_\Delta$ has a linear resolution. Let $c$ be the codimension of $I$ and let $s$ be its generating degree (the degree of any minimal generator of $I$, or the size of any minimal non-face). Then for each $j$ such that $1\leq j\leq c$ we have $$b_j(S/I_\Delta)\geq\binom{s+j-2}{j-1}\sum_{i=j-1}^{c-1}\binom{s+i-1}{s+j-2}.$$   
Furthermore, the following are equivalent:
\begin{enumerate}
\item Equality occurs for some $j$ (between $1$ and $c$). 
\item $\Delta$ is Cohen-Macaulay (so $\Delta$ is bi-Cohen-Macaulay in the sense of \cite{floystad}). 
\item $\Delta$ has exactly $\binom{s+c-1}{c}$ facets  of size $n-c$. 
\item $\Delta$ has exactly $\binom{s+c-1}{s}$ minimal non-faces (necessarily of size $s$). 
\end{enumerate}
\end{theorem}

This allows us to complete classify complexes with linear resolution and type defect zero as follows.  

\begin{corollary}
Let $\Delta$ be a simplicial complex  such that $I=I_\Delta$ has a linear resolution. Then $\td(\Delta)=0$ if and only if $\Delta$ is a tree of $(d-1)$-simplices, glued along codimension one faces. 
\end{corollary}

Another noteworthy consequence is that the inequalities in Theorem \ref{intro2} hold for any homogenous ideals with linear resolution, monomial or not, see Theorem \ref{lin}.  

The article is organized as follows.  Section \ref{s:background} deals with background and preliminaries, such as the connection between type and the Cohen-Macaulay property.  In Section \ref{graphs} we consider the case when $\Delta$ is a simple graph, and show the relationships between type and chordality.  In the final section, we consider the application of type to complexes whose Stanley-Reisner ideals have linear resolutions.

%
%%The article is organized as follows. Section \ref{s:background} deals with preparatory materials. Here we explain the connection between type and Cohen-Macaulayness, Theorem \ref{typefact}, which is really just a reformulation of a famous result in commutative algebra. We also prove the gluing results, Theorem \ref{mainglue} and Corollary \ref{secondglue}, and discuss the modified type of $2$-Cohen-Macaulay complexes, and how the type defect behaves for join of two complexes. 
%%
%%In Section \ref{graphs} we consider the case when  $\Delta$ is a graph, which we view as a 1-dimensional simplicial complex. In particular, we show in Theorem \ref{chordal} that a graph $G$ is chordal if and only if the type defect of every induced subgraph is nonnegative.  Motivated by this surprising connection,  we then characterize the graphs for which every connected subgraph has zero type defect.  This set of graphs, which we call \emph{treeish}, contains all trees, as well as graphs that can be obtained by adding triangles to a tree in a specific way. 
%%
%%We also consider a natural generalization of treeish graphs to simplicial complexes, ones in which certain induced subcomplexes have type defect zero.  
%
%We then turn our attention to higher dimensional versions of our chordality results in Section \ref{linear}. Throughout this section we assume that $I_{\Delta}$ has linear resolution. We prove Theorem \ref{intro2} and its consequences, including the lower bounds for Betti numbers for general graded ideals, Theorem \ref{lin}, here. 

\begin{acknowledgement}
We thank the Mathematics Department at the University of Kansas, where most of this work was initiated and produced, for hospitable working conditions. This project was initially motivated by several conversations between the first author and Volkmar Welker, to whom we are grateful.  We also thank Bruno Benedetti  and Jos\'e Samper for helpful conversations, and Alexander Lazar, whose comments during a seminar talk suggested Proposition \ref{join}. 

Finally, we thank the two anonymous referees for their careful readings.  Their corrections and recommendations greatly improved this article. 

The first author is partially  supported by NSA grant H98230-16-1-001. 
\end{acknowledgement}

\section{Background and Preliminary results}\label{s:background}

Let $\Delta$ be a $(d-1)$-dimensional simplicial complex with vertex set $V$ (where $|V|=n$) and associated Stanley-Reisner ideal $I_\Delta$.  Recall that $I_\Delta$ is the ideal generated by monomials corresponding to non-faces of $\Delta$.  For $W \subseteq V$, we write $\Delta|_W$ to denote the subcomplex of $\Delta$ induced on the set $W$.  For a vertex $v \in V$, we write $\Delta - v$ as short for $\Delta|_{V\setminus \{v\}}$.  We use similar notation when dealing with graphs (indeed, a simple graph is a $1$-dimensional simplicial complex). In the interest of readability, we also write $\dim$ as short for $\dim_k$. 

The \emph{codimension} of $\Delta$ is the number of vertices in the complement of a maximal facet of $\Delta$.  As $\Delta$ is $(d-1)$-dimensional on $n$ vertices, we have: 
\[
\codim(\Delta) = n - d.  
\]

For ease of notation we shall write  $b_{i,j} (M)$ for $\dim \Tor_i^S(M, k)_j$. 
 We first recall Hochster's Formula (see, for instance, \cite{BH}), which expresses the multigraded Betti numbers in terms of the homologies of induced subcomplexes of $\Delta$.  
\begin{theorem}[Hochster's Formula]
The multigraded Betti numbers $b_{i,j}$ of $S/I_\Delta$ are given by:
%\[
%b_{i, j} (I_\Delta) = \sum_{|W| = j} \dim(\tilde{H}_{j-i - 2}(\Delta|_W)). %= \sum_{|W| = j} \dim(\tilde{H}_{i-1}(\lk_{\Delta^\vee}(V \setminus W)))%
%\]
%If we resolve the quotient instead, we get:
\[
b_{i, j} (S / I_\Delta) = b_{i-1, j} (I_\Delta) = \sum_{|W| = j} \dim(\tilde{H}_{j-i - 1}(\Delta|_W)). %= \sum_{|W| = j} \dim(\tilde{H}_{i-2}(\lk_{\Delta^\vee}(V \setminus W)))
\]
\end{theorem}

Hochster's Formula is often used to state algebraic properties in terms of the associated Stanley-Reisner complex. Perhaps the best known of these results is Reisner's Criterion (see, for instance, \cite{BH}):

\begin{theorem}[Reisner's Criterion]
Let $\Delta$ be a simplicial complex and $I_\Delta$ its Stanley-Reisner ideal. Then $S/I_\Delta$ is Cohen-Macaulay (over $k$) iff for every face $F$ of $\Delta$, we have $\dim(\tilde{H}_i(\lk F)) = 0$ for all $i$ less than the dimension of $\lk F$. 
\end{theorem}

Hochster's Formula immediately implies severals facts about the modified type and type defect. 

\begin{corollary} \label{basic} Let $\Delta$ be a $(d-1)$-dimensional complex with $n$ vertices and let $c=n-d$ be its codimension. It is immediate from the definitions that $\type(\Delta) \geq 0$ for any $\Delta$, and that $\type(\Delta) < b_c(S/I_\Delta)$ iff $\Delta$ is a simplex. Furthermore, we have the following:  
\begin{enumerate}
\item $\type(\Delta)= b_{c}(S/I_{\Delta})>0$ if and only if $\Delta$ is not a simplex. 
\item When $\Delta$ is a simplex, $\type(\Delta) = 0$ and $\td(\Delta)=0$. 
\item If $\Delta$ is Gorenstein, $\td(\Delta)=0$ if and only if $\Delta$ is a $(d-1)$-simplex or the boundary of a $d$-simplex. 
\end{enumerate}
\end{corollary}

\begin{proof}
For $(1)$, we need to show that $\dim \Tor_c^S(S/ I_\Delta,k)_c=0$ if and only if $\Delta$ is not a simplex. By Hochster's formula, this is $$b_{c,c}(S/I_{\Delta}) = \sum_{|W| = c} \dim(\tilde{H}_{-1}(\Delta|_W))$$
This sum is $0$ if and only if $\Delta_W$ is non-empty for each subset $W$ of size $c=n-d$. This clearly happens if and only if $n>d$, namely, $\Delta$ is not a simplex. 

Part $(2)$ is clear. For $(3)$, as $S/I_{\Delta}$ is Gorenstein, $b_c(S/I_{\Delta})=1$. Thus either $\Delta$ is a simplex (in which case $\type(\Delta) = 0$ and $c = 0$), or $c=1$, which means $\Delta$ is a boundary of a simplex.  
\end{proof}

\begin{remark}\label{typedefn}
Let $\Delta$ be a $(d-1)$-dimensional simplicial complex on $n$ vertices.  We have

\begin{align}\label{firsthoch}
\type(\Delta) =  \sum_{j >c} b_{c, j }(S/I_\Delta) = \sum_{j >c} \sum_{|W| = j} \dim(\tilde{H}_{j - c -1}(\Delta|_W)).
\end{align}

Substituting $j = n - d + l$ changes Equation \ref{firsthoch} to

\begin{align}\label{main}
\type(\Delta) = \sum_{l > 0} \sum_{|W| = c+l } \dim(\tilde{H}_{l-1} (\Delta|_W)).
\end{align}

We can also use the Alexander dual form of Hochster's Formula to get

\[
\type(\Delta) = \sum_{j> 0} \sum_{|W| = c + j, W \notin \Delta} \dim(\tilde{H}_{c-2}(\lk_{\Delta^\vee} (V \setminus W))).
\]

\end{remark}

\begin{observation}\label{2cmobs}
If $\Delta$ is 2-Cohen-Macaulay and $(d-1)$-dimensional, then
\[
\type(\Delta) = \dim (\tilde{H}_{d-1}(\Delta)).  
\] 

This is simply because of Hochster's formula and the fact that $\tilde H_{j-c-1} ({\Delta|_W})=0$ unless $|W|=n$. Note that $2$-Cohen-Macaulay complexes include matroid complexes and triangulations of spheres.  

\end{observation}

Let $\Delta_1, \Delta_2$ be simplicial complexes with Stanley-Reisner rings $k[\Delta_1]= S_1/I_{\Delta_1}$ and $k[\Delta_2]= S_2/I_{\Delta_2}$ respectively. Let $\Delta$ be the join $\Delta_1 * \Delta_2$. The Stanley Reisner ring of $\Delta$ is $k[\Delta_1]\otimes_k k[\Delta_2]$, a quotient of $S=S_1\otimes_k S_2$.   Then if $P_{\Delta}(t) : = \sum_{i\geq 0} b_i(S/I_\Delta)t^i$ is the Poincar\'e series of $k[\Delta]$, it is easy to see that $$P_{\Delta}(t) = P_{\Delta_1}(t)P_{\Delta_2}(t).$$ 

Now, let $c_1,c_2$ be the codimensions of $\Delta_1, \Delta_2$ respectively. Let $t_1 = b_{c_1}(S_1/I_{\Delta_1})$ and  $t_2 = b_{c_2}(S_1/I_{\Delta_2})$ be the types. Then $c = \codim \Delta = c_1+c_2$ and  the multiplicative formula above for Poincar\'e series tells us that 
$$b_c(S/I_{\Delta}) \geq t_1t_2 $$ 
where equality happens if and only if $\Delta_1, \Delta_2$ are Cohen-Macaulay. It follows immediately that:
\begin{proposition}\label{join}
Let $\Delta = \Delta_1 * \Delta_2$. Then: 
\begin{enumerate}
\item $\td(\Delta) = \td(\Delta_1)$ if $\Delta_2$ is a simplex. 
\item $\td(\Delta) \geq  \td(\Delta_1)+ \td(\Delta_2)-1$. Equality happens if and only if  $\Delta_1, \Delta_2$ are Cohen-Macaulay and one of them is Gorenstein but not a simplex. 

\end{enumerate}
\end{proposition}

A rather non-trivial fact about type defect is the following.  

\begin{theorem}\label{typefact}
If $\td(\Delta) \leq 0$ (i.e., if $\type(\Delta) \leq \codim(\Delta)$), then $\Delta$ is Cohen-Macaulay.  
\end{theorem}

\begin{proof}
If $\Delta$ is a simplex, then it is Cohen-Macaulay. Otherwise, the modified type is just $b_c(S/I_{\Delta})$, and the assertion  is a consequence of the classical Syzygy Theorem, see for example \cite[Corollary 9.5.6]{BH}. 
\end{proof}

It's worth noting that the converse to the above is false. Using Equation \ref{gengraphtype} from Section \ref{graphs}, it's easy to see that $\td(K_n) = \binom{n}{2} -2n + 3$, where $K_n$ is the complete graph on $n$ vertices, viewed as a $1$-dimensional complex. So there exist Cohen-Macaulay complexes with arbitrarily high type defect. 

\begin{observation}\label{obsCM}
In the glueing result below we need to consider  complexes inside a bigger set of vertices. If $\Delta$ is Cohen-Macaulay, one can compute $\type(\Delta)$ using $\Tor_{c'}^{S'}(k[\Delta], k)$ with $S= k[x_1,\dots, x_n, y_1,\dots, y_m]$ where the $y_i$s are vertices not in $\Delta$. This is because $\Tor_c$ is the last non-zero $\Tor$, and this number is unchanged if we throw in more variables. Note that this is not true if  we drop the Cohen-Macaulayness assumption. For example, let $\Delta$ be the union of disjoint  edges  $\{xy\}$ and  $\{zw\}$.  Then $k[\Delta]= k[x,y,z,w]/(xz,xw, yw,yw)$ with $c=2$ and $\type(\Delta) = 4$.  If we introduce a new vertex $t$, then we have $S'= k[x,y,z,w,t]$,  $k[{\Delta}] =  S'/(xz, xw, yw,yw,t)$, the new codimension is $c'=3$, and $\dim \Tor_3^{S'}(k[\Delta], k) = 5$.  This is the main reason why in Theorem \ref{mainglue} we need to assume the complexes are Cohen-Macaulay. 
\end{observation}

Our next result concerns how the type and type defect behave when Cohen-Macaulay complexes are glued together.  It gives a more precise version of   \cite[Lemma 3.5]{murai}. We need this extension, as we need to know what happens when one of the complexes glued is a simplex.

If $\Delta_1$ and $\Delta_2$ are two complexes on disjoint vertex sets, and $E_1$ and $E_2$ are faces of $\Delta_1$ and $\Delta_2$, respectively, we write $\Delta\sqcup_{E_1= E_2} \Delta'$ to denote the complex obtained by identifying $E_1$ and $E_2$.  

\begin{theorem}\label{mainglue}
Let $\Delta_1$ and $\Delta_2$ be Cohen-Macaulay complexes of dimension $(d-1)$, and let $E_1$ and $E_2$ be $(\ell-1)$-dimensional faces of $\Delta_1$ and $\Delta_2$, respectively.  Then 
\[
\type(\Delta \sqcup_{E_1 = E_2} \Delta') = \type(\Delta_1) + \type(\Delta_2) + \binom{n-\ell}{c+1}
\]
Here $c= n-d$, where $n$ is the number of vertices in $\Delta \sqcup_{E_1 = E_2} \Delta'$.

\end{theorem}

\begin{proof}
Consider the short exact sequence 
\[
0 \to k[\Delta_1 \cup \Delta_2] \rightarrow k[\Delta_1] \oplus k[\Delta_2] \rightarrow k[\Delta_1 \cap \Delta_2] \to 0. 
\]
Let $\Sigma = \Delta_1 \cup \Delta_2$, and let $\Gamma = \Delta_1 \cap \Delta'_2$ (note that $\Gamma$ is just the simplex resulting from identifying $E_1$ and $E_2$). Taking $\Tor^S(-, k)$ of this sequence yields:
\[
\cdots \rightarrow \Tor^S_{i+1}(k[\Gamma]) \rightarrow \Tor_i^S(k[\Sigma]) \rightarrow \Tor^S_i(k[\Delta_1]) \oplus \Tor^S_i(k[\Delta_2]) \rightarrow 
\Tor^S_i(k[E]) \rightarrow \cdots
\]
As  $\Tor^S_i(k[\Delta_1]) \oplus \Tor^S_i(k[\Delta_2]) $ is zero for all $i>c$, we have an exact sequence:
\[
 0\rightarrow \Tor^S_{c+1}(k[\Gamma]) \rightarrow \Tor^S_{c}(k[\Sigma_1]) \rightarrow \Tor^S_{c}(k[\Delta_1]) \oplus \Tor^S_{c}(k[\Delta_2]).
\]

Now we argue in as in \cite{murai}. $\Tor^S_i(k[\Gamma]) $ only exists in degree $i$, as $\Gamma$ is a simplex. On the other hand, by  Hochster's formula, for any complex $\Delta_1$ with more than one facet, $\Tor_c^S(k[\Delta_1])_c =0$ (where $c$ is the codimension). So the above is a short exact sequence (that is, the last map is surjective) in every degree $j\neq c$, and in degree $c$ it becomes:
\[
0 \to 0 \to 0 \to \Tor^S_{c}(k[\Delta_1])_c \oplus \Tor^S_{c}(k[\Delta_2])_c
\] 
Which implies that $\dim \Tor^S_{c}(k[\Sigma])$ equals
\[
\dim \left[\Tor^S_{c+1}(k[\Gamma])\right] + \dim \left[\Tor^S_{c}(k[\Delta_1]) \oplus \Tor^S_{c}(k[\Delta_2])\right] - \dim \left[\Tor^S_{c}(k[\Delta_1])_c \oplus \Tor^S_{c}(k[\Delta_2])_c\right].
\]

Since $\Gamma$ is just the simplex $E_1 = E_2$, note that $k[\Delta_1 \cap \Delta_2] = S/( x_i : x_i \notin \Gamma )$.  This quotient is resolved by the Koszul resolution, meaning that the dimension of $\Tor^S_{c+1}$ is $\binom{n-\ell}{c+1}$. Finally, note that as $\Delta_1$ is  Cohen-Macaulay, 
$$\dim \left[\Tor^S_{c}(k[\Delta_1])\right] -  \dim \left[\Tor^S_{c}(k[\Delta_1])_c\right]=\type(\Delta_1),$$
and the same equality holds for $\Delta_2$ (see \ref{obsCM}).
\end{proof}

\begin{corollary}\label{secondglue}
Let  $\Delta_1$ and $\Delta_2$ be Cohen-Macaulay complexes of dimension $d-1$.  Let $E_1$ and $E_2$ be $(\ell-1)$ dimensional faces of $\Delta_1$ and $\Delta_2$, respectively.  Then 
\[
\td(\Delta \sqcup_{E_1 = E_2} \Delta') = \td(\Delta) + \td(\Delta') - d + \ell + \binom{n-\ell}{c+1}.
\]
In particular, if $\ell\geq d-1$, then
$$\td(\Delta \sqcup_{E_1=E_2} \Delta') = \td(\Delta) + \td(\Delta').$$

\end{corollary}

\begin{proof}
The assertions  follow from  Theorem \ref{mainglue}. One just needs to take care of the codimensions in the definition of type defect: if  $\Delta$ has $p$ vertices and $\Delta'$ has $m$ vertices, then $n = p + m - \ell$, so the codimensions are $p-d, m-d, p+m-\ell-d$ respectively.  

\end{proof}

\section{Chordal graphs and treeish complexes}\label{graphs}

Throughout this section, let $\Delta = G$ be a simple graph on $n$ vertices, viewed as a $1$-dimensional simplicial complex.  In this case we can simplify Equation \ref{main} in Remark \ref{typedefn} as follows.

\[
\type(G) =  \dim(\tilde{H_1}(G))  +  \sum_{|W| = n -1} \dim(\tilde{H}_0(G_W)) . 
\]

Note that the term $\dim(\tilde{H_1}(G))$ is simply the number of basic cycles of $G$.  Assume $G$ is connected.  Let $e$ be the number of edges of $G$.  Then any spanning tree of $G$ has $n-1$ edges, and any edge not in a given spanning tree corresponds to a basic cycle of $G$.  Thus, $\dim(\tilde{H}_1(G)) = e - n + 1$.  For any graph $H$, let $C(H)$ be the number of connected components of $H$.  Also, if $W$ is a subset of vertices of $G$ with $|W| = n-1$, let $W = V(G) \setminus \{v\}$.  Then $\dim(\tilde{H}_0(G - v)) = C(G - v) - 1$.  Thus, the above expression becomes:

\begin{align}\label{graphtype}
\type(G) = e - n + 1 + \sum_{v \in V(G)} (C(G - v) - 1),
\end{align}

provided $G$ is connected.

More generally, for any simple graph $G$, the dimension of $\tilde{H}_1(G)$ is given by $e - n + C(G)$, and so 
\begin{align}\label{gengraphtype}
\type(G) = e - n + C(G) + \sum_{v \in V(G)} (C(G - v) - 1),
\end{align}

\begin{observation}\label{2conn}
If $G$ is \emph{$2$-connected}, meaning the removal of any one vertex cannot disconnect the graph, we have $\type(G) = e - n +1$.  Note that this can also be seen via Observation \ref{2cmobs}.  Thus, If $G$ is a $2$-connected graph with $e$ edges and $n$ vertices, then $\td(G) = (e - n + 1) + 1 - n + 1 = e - 2n +3$.   
\end{observation}

\begin{proposition}\label{treetd}
If $T$ is a tree, then $\td(T) = 0$.  
\end{proposition}

\begin{proof}
Let $v \in T$, and suppose $T$ has $n$ vertices and $e$ edges.  Then the number of connected components of $G - v$ is $\deg(v)$.  Since $T$ has no cycles, 
\[
\type(T) = \sum_{v \in V(T)} (\deg(v) - 1) = 2e - n = 2(n-1) - n = n -2.  
\]
As $\codim(T) = n - 1 - 1 = n-2$, we get $\td(T) = 0$. 
\end{proof}

\noindent
\emph{Chordality} is one of the most widely-studied graph properties.  It turns out that chordality is equivalent to a requirement on the type defect, as shown by the next theorem. 

\begin{theorem}\label{chordal}
$G$ is chordal iff $\td(G|_W) \geq 0$ for any $W \subseteq V$.  
\end{theorem}

Our proofs use the following classical result of Dirac: 

\begin{theorem}\cite{dirac}\label{dirac}
Let $G$ be a chordal graph.  Then there exists a vertex $v$ of $G$ such that $G|_{N(v)}$ is complete, where we write $N(v)$ to mean the set of neighbors of $v$.  Such a vertex is called \emph{simplicial}.  
\end{theorem}

\begin{proof}[Proof of Theorem \ref{chordal}]
First assume that $G$ is a chordal graph, let $v$ be a simplicial vertex of $G$, and let $N(v)$ be its neighbors.  Then $\td(G - v) \geq 0$, by induction, the base case being immediate.  Adding $v$ to $G\setminus v$ increases the codimension by $1$, so we need to show that the type increases by at least $1$.  

If $v$ has only one neighbor $w$, then the number of components of $G - w$ is one greater than the number of components of $(G - v) - w$.  Thus $\type(G) = \type(G - v) + 1$.  

If $v$ has at least two neighbors, then $\dim \tilde{H}_1(G) \geq \dim \tilde{H}_1(G - v) + 1$.  

For the converse, assume $G$ is not chordal.  Then there exists a set $W \subseteq V$ such that $G|_W$ is a cycle and $|W| \geq 4$.  Using Observation \ref{2conn}, $\td(G|_W) = 1 - (|W| -2) < 0$. 
\end{proof}

\begin{corollary}
$G$ is chordal iff $\td(G|_W) \geq 0$ for any $W \subseteq V$ such that $G|_W$ is connected.  
\end{corollary}

\begin{proof}
It is straightforward to see that the type defect of a graph is greater than or equal to the sum of the type defects of its connected components.  Thus $\td(G|_W) \geq 0$ for all induced subgraphs $G|_W$ of $G$ if and only if this holds for all connected subgraphs $G|_W$, and the result follows from Theorem \ref{chordal}.
\end{proof}

Now we turn our attention to graphs for which every connected induced subgraph has type defect zero.  By Theorem \ref{chordal}, such graphs must necessarily be chordal.  However, note that many chordal graphs have positive type defect (for example, $K_n$ for $n > 3$).  

\begin{definition} \label{treeishdefn}
Call a graph \emph{treeish} if it can be constructed recursively via the following rules.  
\begin{enumerate}
\item An edge or a triangle is treeish. 
\item If $G$ is treeish and $v$ is a vertex of $G$, adding a new vertex $w$ and a new edge $\{v, w\}$ results in a treeish graph. 
\item If $G$ is treeish, and $\{v, w\}$ is an edge of $G$, then the graph obtained from $G$ by adding a new vertex $z$ and the edges $\{v, z\}$ and $\{w, z\}$ is treeish.  
\end{enumerate}
\end{definition}

Clearly any tree is treeish, as it can be constructed from a single vertex by repeated application of rule (2).  Also, any treeish graph is easily seen to be chordal.

\begin{figure}[htp]
\centering
\includegraphics[height = 1.6in]{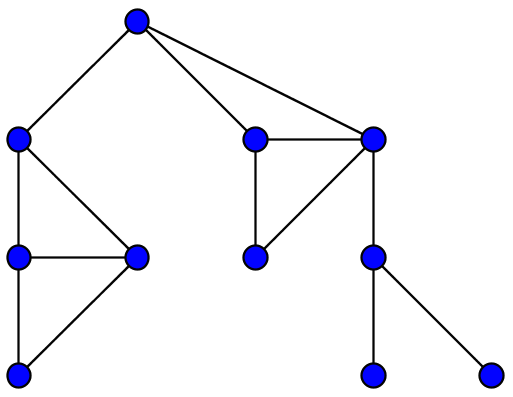}
\caption{A treeish graph.}
\end{figure}

\begin{theorem}\label{tfae}
The following are equivalent. 
\begin{enumerate}
\item $G$ is treeish.
\item $\td(G|_W) = 0$ for any $W \subseteq V$ such that $G|_W$ is connected.
\item $G$ is chordal, and the number of triangles of $G$ is the dimension of $\tilde{H}_1(G)$. 
\end{enumerate}
\end{theorem}

\begin{proof}
(1) $\Rightarrow$ (2): 

Let $W$ be a subset of vertices such that $G|_W$ is connected.  Then $G|_W$ is easily seen to be treeish; simply restrict the operations described in Definition \ref{treeishdefn} to the vertices in $W$ (in doing so an operation of type (3) may become an operation of type (2), but the subgraph is still treeish).  So, it suffices to show that $\td(G) = 0$ for any treeish graph.  However, this follows immediately from Theorem \ref{mainglue}.

%We use Equation \ref{graphtype} to compute the type of $G$ by induction on the number of vertices.  If $G$ has two vertices, then it must contain exactly one edge.  Equation \ref{graphtype} gives $\type(G) = 0$, and $\codim(G) = 0$, so $\td(G) = 0$.  
%
%Next we show that the operations described in Definition \ref{treeishdefn} preserve $\td(G)$.  Let $G'$ be the result of applying Operation (2) to $G$.  Then clearly $\codim(G') = \codim(G) + 1$.  Moreover we have $C(G - x) = C(G' - x)$ for any vertex $x \neq v$ of $G'$, and $C(G'-v) = C(G - v) + 1$, since the graph $G' - v$ contains one more connected component than $G - v$, namely the isolated vertex $w$.  Thus the sum in Equation \ref{graphtype} increases by one as we pass from $G$ to $G'$.  Since $G'$ has one more vertex and one more edge than $G$, we have $\type(G') = \type(G) +1$, meaning $\td(G') = \td(G) = 0$.  The proof for Operation (3) is very similar.  In this case, we add one new vertex and two new edges to $G$ to obtain $G'$.  Note that $C(G - x) = C(G'- x)$ for any vertex $x$ of $G$, and that $C(G' - z) = 1$.  Thus, $\type(G') = \type(G) + 1$, and so $\td(G') = \td(G) = 0$.  

(2) $\Rightarrow$ (3):  

First suppose that $G$ is not chordal.  Then there exists a subset $W$ of vertices of $G$ such that $|W| \geq 4$ and $G|_W$ is a cycle.  Then Observation \ref{2conn} gives $\td(G|_W) = 1 - (|W| -2) < 0$.  So we can assume that $G$ is chordal.  Let $v$ be a simplicial vertex of $G$ as guaranteed by Theorem \ref{dirac}.  By induction on the number of vertices of $G$ (the base case being straightforward), we have that $G - v$ is chordal, and that $\dim(\tilde{H}_1(G - v))$ is the number of triangles of $G - v$.  Let $W$ be the set of neighbors of $v$.  We claim that $|W| = 1$ or $2$.  Indeed, suppose $|W| \geq 3$.  Then $G|_{W \cup \{v\}}$ would be a complete graph on $m \geq 4$ vertices, meaning 
\[
\td(G|_{W \cup \{v\}}) = \binom{m}{2} - m + 1 - (m - 2) = \frac{m(m-1)}{2} - 2m + 3 = \frac{m^2 - 5m + 6}{2} > 0 \text{ if } m \geq 4.
\]
Thus we may assume that $|W| = 1$ or $2$.  If $|W| = 1$, then $\dim(\tilde{H}_1(G)) = \dim(\tilde{H}_1(G - v))$, and $G$ clearly has the same number of triangles as $G - v$.  If $|W| = 2$, then $G$ has one more triangle than $G - v$ (namely $G_{W \cup \{v\}}$).  However, this triangle cannot be written as a sum of any other elements in $\tilde{H}_1(G)$, as it is the only triangle containing the vertex $v$.  Thus $\dim(\tilde{H}_1(G)) = \dim (\tilde{H}_1(G - v))  + 1$.  

(3) $\Rightarrow$ (1):  

Here we use Theorem \ref{dirac} in the same way as in the previous case.  Indeed, let $v$ be a simplicial vertex of the graph $G$, and let $W$ be its neighbors.  If $|W| \geq 3$, then it is easy to see that $\dim(\tilde{H}_1(G))$ is strictly less than the number of triangles of $G$:  If $\{a, b, c\} \subseteq W$, then the homology element corresponding to the boundary of the triangle $abc$ is a linear combination of the boundaries of the triangles $abv, acv,$ and $bcv$.  Thus we have $|W|  = 1$ or $2$.  By induction, $G - v$ is treeish.  But then the addition of $v$ and the corresponding edge(s) either corresponds to rule (2) or rule (3), meaning $G$ is treeish. 
\end{proof}

%\begin{proof}
%First, the implication (1) $\Rightarrow$ (3) is immediate from Definition \ref{treeishdefn}.  
%
%For the implication (1) $\Rightarrow $ (2), note that any connected subgraph of a treeish graph is also treeish: Simply restrict the construction of the treeish graph to the subset $W$ of vertices.  Thus, it suffices to show that every treeish graph $G$ satisfies $td(G) = 0$.  If $G$ is a tree, then Proposition \ref{treetd} gives $\td(G) = 0$.  Now we show, using Equation \ref{graphtype}, that rule (3) in Definition \ref{treeishdefn} applied to a graph $G$ preserves the type defect of $G$.  Suppose rule (3) is applied to $G$, and call the resulting graph $G'$.  Then $\codim(G') = \codim(G) + 1$, and $\rank(\tilde{H}_1(G')) = \rank(\tilde{H}_1(G)) + 1$.  Moreover, $|C(G' - z)| = 1$, and $|C(G' - x)| = |C(G - x)|$ for any vertex $x$ of $G$.  Thus $\td(G') = \td(G)$.  
%
%\end{proof}

%Let $\C$ be a clutter and let $e$ be a simplicial submaximal circuit of $\C$.  Then $\Cl(\C\setminus e)$ is the deletion of all faces from $\Cl(\C)$ containing $e$.  
%
%

\begin{theorem}
Let $\Sigma$ be a flag complex with vertex set $V$.  Then $\Sigma$ is the clique complex of some chordal graph if and only if $\td(\Sigma|_W) \geq 0$ for every $W \subseteq V$.  
\end{theorem}

\begin{proof}
First suppose that $\Sigma$ is the clique complex of some chordal graph $G$. If $\Sigma$ is not Cohen-Macaulay, Theorem \ref{typefact} gives that $\td(\Sigma) \geq 0$.  So we can assume that $\Sigma$ is Cohen-Macaulay.  Let $v$ be a simplicial vertex of $G$ and $N(v)$ its neighbors.  If we set $\Delta = \Sigma - v$ and let $\Delta'$ denote the full simplex on $\{v\} \cup N(v)$, then $\Sigma = \Delta \cup \Delta'$, and $\Delta \cap \Delta'$ is the simplex on $N(v)$ (If $\Delta$ is a full simplex then $\Sigma$ is empty, but the result follows immediately in this case).  Because $\Delta'$ is a facet of $\Sigma$, which is Cohen-Macaulay, it follows that $\Delta$ and $\Delta'$ are of the same dimension.  Since $\Delta$ is an induced subcomplex of $\Sigma$, $\Delta$ must be Cohen-Macaulay as well.  By induction, each of $\Delta$ and $\Delta'$ must have nonnegative type defect.  Thus, by Corollary \ref{secondglue}, $\td(\Sigma) = \td(\Delta) + \td(\Delta') \geq 0$.  Now let $W \subseteq V$.  Then $\Sigma_W$ is the clique complex of $G|_W$, which is chordal. 

For the converse, let $G$ be a graph that is not chordal.  Then $G$ must contain a chordless cycle on vertex set $W$, where $|W| \geq 4$.  Then, as in the proof of Theorem \ref{chordal}, $\td(\Sigma|_W) \leq -1$.  
\end{proof}  

A natural question is how to extend the treeish graphs of Definition \ref{treeishdefn} to arbitrary simplicial complexes.  The most natural generalization is immediate, given Theorem \ref{mainglue} and Observation \ref{2cmobs}.  

\begin{observation}
Let $\Delta$ be a triangulation of a $(d-1)$-dimensional sphere.  If $\td(\Delta) = 0$, then $\Delta$ must be the boundary of the $d$-simplex. 
\end{observation}

\begin{proof}
As spheres are $2$-Cohen-Macaulay, Observation \ref{2cmobs} gives that $\type(\Delta) = \dim \tilde{H}_{d-1}(\Delta) = 1$.  If $\td(\Delta) = 0$, then the codimension of $\Delta$ must be $1$.  But clearly this can only happen if $\Delta$ is the boundary of the $d$-simplex.  
\end{proof}

\begin{definition}\label{treeishcomplexdefn}
We call a pure $(d-1)$-dimensional simplicial complex \emph{treeish} if it can be constructed recursively via the following rules. 
\begin{enumerate}
\item A single $(d-1)$-simplex or a boundary of a $d$-simplex is treeish.  
\item If $\Sigma$ is treeish and $F$ is a $(d-1)$-simplex, then $\Sigma \sqcup_{E' = E} F$ is treeish, where $E'$ and $E$ are $(d-2)$-dimensional faces of $\Sigma$ and $F$, respectively. 
\item If $\Sigma$ is treeish and $\Delta$ is the boundary of a $d$-simplex, then $\Sigma \sqcup_{E' = E} \Delta$ is treeish, where $E'$ and $E$ are $(d-2)$ or $(d-1)$-dimensional faces of $\Sigma$ and $\Delta$, respectively.  
\end{enumerate}
\end{definition}

Note that it follows from Theorem \ref{mainglue} and the above observation that any treeish complex $\Sigma$ satisfies $\td(\Sigma) = 0$.  Moreover, as in the graph case, the treeish property is somewhat hereditary, as shown by the following theorem.  Recall that a pure simplicial complex is called \emph{strongly facet-connected} if, for any two facets $F$ and $F'$, there is a chain of facets $F = F_0, F_1, F_2, \ldots, F_t = F'$ such that $F_i$ intersects $F_{i+1}$ in a codimension-one face for each $i$.    

\begin{theorem}
Let $\Sigma$ be a treeish complex, and let $\Delta$ be an induced subcomplex of $\Sigma$ that is strongly facet-connected.  Then $\Delta$ is a treeish complex.  
\end{theorem}

\begin{proof} 
We think of building $\Sigma$ via the steps in Definition \ref{treeishcomplexdefn}.  So, suppose $\Sigma$ has been constructed via subcomplexes as $\Sigma = \Sigma_1 \cup  \Sigma_2 \cup \cdots \cup \Sigma_t$, where each $\Sigma_i$ is either a full simplex or a simplex boundary, identified along a facet or a codimension-1 face.  
 We proceed by induction on $t$. Let $W$ be the set of vertices of $\Delta$ (so that $\Delta = \Sigma_W$).  If $t=1$, then $\Delta$ is either a simplex or a boundary of a simplex. Suppose $t>1$. Let $\Sigma' =  \Sigma_1 \cup  \Sigma_2 \cup \cdots \cup \Sigma_{t-1}$ and $\Delta' = \Sigma'|_W$ (here we abuse notation, as $W$ is not contained in the vertex set of $\Sigma'$, though it is clear what we mean). By induction $\Delta'$ is treeish. 
 
 By construction, $\Sigma$ is obtained from $\Sigma'$ by adding one or two vertices and taking a cone of the vertex (or boundary of the cone, in the case of two points) over some face $F$ in $\Sigma'$. Thus $\Delta$ is obtained from $\Delta'$ by adding at most two vertices and performing similar operations over some face $F'= F \cap \Delta'$. By the assumption of $\Delta$ being  strongly facet-connected, $F'$ must be of codimension one in  $\Delta'$, and we are done. \end{proof}

\begin{corollary}
If $\Sigma$ is a treeish complex, then $\td(\Delta) = 0$ for any strongly facet-connected induced subcomplex $\Delta$ of $\Sigma$.  
\end{corollary}

Note that this generalizes the (1) $\Rightarrow$ (2) direction of Theorem \ref{tfae}, as a graph is connected if and only if it is strongly facet-connected.  

\begin{example}
In the above Theorem and Corollary, it is not enough to assume that $\Delta$ is pure and connected. To see this, consider the complex with facets $\{1,2,3\}$, $\{2,3,4\}, \{3,4,5\}, \{3,5,6\}$ and induce on the vertex set $W=\{1,2,3,5,6\}$.   
\end{example}

\section{Ideals with linear resolution}\label{linear}

In this section we study lower bounds on Betti numbers of a complex with linear resolution.  We begin with a proof of Theorem \ref{intro2}. 

%\begin{theorem}\label{linres}
%Let $\Delta$ be a simplicial complex  such that $I=I_\Delta$ has a linear resolution. Let $c$ be the codimension of $I$ and $s$ be its generating degree (that is, the degree of any minimal generator of $I$, or the smallest size of a minimal non-face). Then for each $j$ such that $1\leq j\leq c$ we have $$b_j(S/I_\Delta)\geq\binom{s+j-2}{j-1}\sum_{i=j-1}^{c-1}\binom{s+i-1}{s+j-2}.$$   
%Furthermore, the following are equivalent:
%\begin{enumerate}
%\item Equality occurs for some $j$ (between $1$ and $c$). 
%\item $\Delta$ is Cohen-Macaulay. 
%\item $\Delta$ has exactly $\binom{s+c-1}{c}$ maximal facets  of size $n-c$. 
%\item $\Delta$ has exactly $\binom{s+c-1}{s}$ minimal non-faces. 
%\end{enumerate}
%\end{theorem}

\begin{proof}[Proof of Theorem \ref{intro2}]
Let $(h_0, h_1,\dots )$ be  the $h$-vector of the Alexander dual of $\Delta$, $\Delta^{\vee}$. By Theorem 4 of \cite{eagonreiner}, $b_j(S/I_\Delta)$ is the coefficient of $t^{j-1}$ in the polynomial
\[
\sum_{i \geq 0} h_i(1+t)^i, 
\]
which is easily seen to be
\[
\sum_{i \geq j-1} \binom{i}{j-1} h_i.
\]

Let $d-1 = \dim \Delta^\vee$. Obviously, $s= n-d$. Since $A= S/I_{\Delta^\vee}$ is Cohen-Macaulay, the $h$ vector of $\Delta^\vee$ is the same as the Hilbert function of the Artinian algebra $B=A/(l_1,\cdots, l_d)$, where the $l_i$s form a linear system of parameters on $A$ (WLOG one can assume that $k$ is infinite). That is, $h_i = \dim B_i$.  We can write $B=S'/J$, where $S'$ is a polynomial ring in $n-d=s$ variables. But since $I_{\Delta^\vee}$ is generated in degree $c$ or higher, the same holds for $J$, so for $i<c$, the $i$-th degree part of $B$ is the same as that of $S'$. Thus $h_i = \binom{s+i-1}{i}$ for $i<c$. Thus 
\[
b_j(S/I_\Delta) \geq  \sum_{i \geq j-1}^{c-1} \binom{i}{j-1} \binom{s+i-1}{i} = \binom{s+j-2}{j-1}\sum_{i=j-1}^{c-1}\binom{s+i-1}{s+j-2}
\]

We now prove the last claim. From the proof, it is clear that any equality occurs if and only if $h_i=0$ for $i\geq c$. But since $J$ is generated in degree $c$ or higher, this is equivalent to $h_c=0$. Which is, in turn, equivalent to $b_{c+1}(S/I)=0$. But this happens if and only if $\pd_S S/I=c$, which is equivalent to $S/I$ being Cohen-Macaulay.  

For $(3)$, note that $h_c=0$ is equivalent to  $J_c= S'_c$, meaning the minimal number of generators of $J$ equals $\dim S'_c = \binom{s+c-1}{c}$. But the minimal number of generators of $J$ is the same as the minimal number of generators of $I_{\Delta^\vee}$, and the latter is simply the number of facets of $\Delta$. Lastly, the equivalence of $(4)$ to the other conditions can be shown by simply noting that $(4)$ is equivalent to the equality for $b_1(S/I_\Delta)$,  which is the number of minimal generators of $I_\Delta$, and this is the number of minimal non-faces. 
\end{proof}

\begin{remark}
Theorem \ref{intro2} gives precise conditions for when $\Delta$ both has a linear resolution and is Cohen-Macaulay. These have been studied in the literature, for example \cite{floystad, hr}. For example, the Bi-Cohen-Macaulay bipartite graphs described in \cite{hr} satisfy the conditions on the number of facets or minimal non-faces given in Theorem \ref{intro2}.  
\end{remark}

\begin{corollary}
Let $\Delta$ be a simplicial complex  such that $I=I_\Delta$ has a linear resolution. Let $c$ be the codimension of $I$. Then for each $j$ such that $1\leq j\leq c$ we have $$b_j(S/I_\Delta)\geq j\binom{c+1}{j+1}.$$ In particular $\td(\Delta)\geq  0$. Furthermore, equality occurs for some $j$ if and only if $\Delta$ is  the clique complex of a chordal graph with exactly $c+1$ maximal cliques of size $n-c$.
\end{corollary}

\begin{proof}
Since $s$ is  the minimal size of a non-face of $\Delta$, it is obvious that $s\geq 2$.  Plugging in $s=2$ in the inequality of the Theorem \ref{intro2}, we have for each $j$ such that $1\leq j\leq c$: $$b_j(S/I_\Delta)\geq j\sum_{i=j}^c\binom{i}{j}=j\binom{c+1}{j+1}.$$
The equality occurs for some $i$ if and only if $s=2$ and the complex $\Delta$ has $c+1$ facets. As $s=2$, we know that $I= I_\Delta$ is an edge ideal, so $\Delta$ must be the clique complex of a chordal graph $G$ by Fr\"oberg's Theorem \cite{froberg}. 
\end{proof}

One can classify the situation of equality in the previous corollary precisely. Call a complex {\it facet constructible} if either $\Delta$ is a simplex or $\Delta = \Delta_1 \sqcup_{E = E'} \Delta_2$ where $\Delta_1$ is facet constructible, $\Delta_2$ is a simplex and $E$ is a facet of its boundary. 

\begin{theorem}
Let $G$ be a chordal graph on $n$ vertices and $\Delta$ be its clique complex of dimension $d-1$. The following are equivalent:
\begin{enumerate}
\item $\Delta$ has with exactly $n-d+1$ maximal facets of size $d$.
\item $G$ has $\frac{(d-1)(2n-d)}{2}$ edges.  
\item $\Delta$ is Cohen-Macaulay. 
\item $\Delta$ is facet constructible. 
\item $h_2(\Delta)\geq 0$. 
\item $\Delta$ is $(S_2)$ in the sense of \cite{mt}.
\item $\Delta$ is shellable.  
\end{enumerate}
\end{theorem}

\begin{proof}
The equivalence of $(1),(2),(3)$ is just Theorem \ref{intro2}. To prove $(4)$ implies $(1)$, we use induction on the number of simplices to show that the skeleton of $\Delta$ is chordal with the correct number of facets. To prove that $(1)$ implies $(4)$  we can  use induction again. Let $G$ be a chordal graph with exactly $n-d+1$ maximal cliques of size $d$. Then there is a simplicial vertex $v$, which has to belong to some maximal clique. Removing $v$ we get a chordal graph with $n-d$ maximal cliques, so the induction hypothesis applies. Note that $(6)$ implies $(5)$ by the main result of \cite{mt},  $(5)$ implies $(3)$ by Corollary \ref{h_s}, and $(3)$ implies $(6)$ by definition.  Finally, to see that $(7)$ is equivalent to the others, it is clear that any facet constructible complex is shellable, and it is a standard result that any shellable complex is Cohen-Macaulay.  
\end{proof}

It is not hard to rewrite the conditions of Theorem \ref{intro2} in terms of the $h$-vector or $f$-vector of $\Delta$. We explain how to do it for $h$-vectors. The linear resolution of $I$ necessarily looks like:
$$ 0 \to S(-s-l+1)^{b_l}  \to \cdots \to S(-s-1)^{b_2}  \to S(-s)^{b_1} \to S \to S/I \to 0 $$
so the alternating sum of the Hilbert series must be $0$, which gives:
$$\frac{1-b_1t^s +b_2t^{s+1}-\dots+(-1)^lt^{s+l-1}}{(1-t)^n} = \frac{\sum h_i(\Delta)t^i}{(1-t)^{n-c}}  $$
From this one gets $h_i(\Delta) = \binom{c+i-1}{i}$ for $i <s$. After that one gets $h_s= \binom{c+s-1}{s}-b_1$, etc. As a consequence, we have, for instance:

\begin{corollary}\label{h_s}
Let $\Delta$ be a simplicial complex  such that $I=I_\Delta$ has  linear resolution. Let  $s$ be the generating degree of $I$. Then $h_s(\Delta)\leq 0$, and equality occurs if and only if $\Delta$ is Cohen-Macaulay.  
\end{corollary}

Finally, we give our main application outside of the monomial situation, showing that the inequalities in Theorem \ref{intro2} extend to all homogenous ideals. 

\begin{theorem}\label{lin}
Let $k$ be a field of any characteristic and  $I \subset S= k[x_1,\dots,x_n]$  a homogenous ideal with a linear resolution. Let $c$ be the codimension of $I$ and $s$ be its generating degree (the degree of any minimal generator of $I$). Then for each $j$ such that $1\leq j\leq c$ we have $$b_j(S/I)\geq\binom{s+j-2}{j-1}\sum_{i=j-1}^{c-1}\binom{s+i-1}{s+j-2}.$$   

\end{theorem}

\begin{proof}
We follow the proof of \cite[Theorem 3.3]{bhpz}. Starting with a homogenous ideal with linear resolution, one can arrive at a square-free monomial ideal with the same Betti sequence via 
taking generic initial ideals and stretching. Note that this works over any field as explained in {\it loc. cit.} Also, it is easy to see that the process does not impact the numbers $c$ and $s$. So the result follows from Theorem \ref{intro2}. 
\end{proof}

\end{document}